\documentclass{article}
\usepackage{amsmath}
\usepackage{amsthm}
\usepackage{mathtext}
\usepackage{amssymb}
\usepackage{tocenter}

\ToCenter{450pt}{650pt}

\newtheorem{thm}{Theorem}
\newtheorem*{thm*}{Theorem}

\newtheorem{lem}{Lemma}
\newtheorem*{lem*}{Lemma}
\newtheorem{prop}{Proposition}
\newtheorem*{con*}{Conjecture}
\theoremstyle{definition}
\newtheorem{defn}{Definition}
\newtheorem*{ex*}{Example}
\newtheorem{cons}{Construction}
\newtheorem*{cons*}{Construction}
\newtheorem{rem}{Remark}
\theoremstyle{remark}
\newtheorem*{not*}{Notation}

\DeclareMathOperator{\conv}{conv}

\DeclareMathOperator{\des}{des}
\DeclareMathOperator{\Sym}{Sym}

\begin{document}

\title{Upper and lower bound theorems for graph-associahedra.}%
\author{V.~M.~Buchstaber \and V.~D.~Volodin}%
\date{}

\maketitle
\begin{abstract}
From the paper of the first author it follows that upper and lower bounds for $\gamma$-vector of a simple polytope imply the bounds for its $g$-,$h$- and $f$-vectors. In the paper of the second author it was obtained unimprovable upper and lower bounds for $\gamma$-vectors of flag nestohedra, particularly, Gal's conjecture was proved for this case. In the present paper we obtain unimprovable upper and lower bounds for $\gamma$-vectors (consequently, for $g$-,$h$- and $f$-vectors) of graph-associahedra and some its important subclasses. We use the constructions that for an $(n-1)$-dimensional graph-associahedron $P_{\Gamma_n}$ give the $n$-dimensional graph-associahedron $P_{\Gamma_{n+1}}$ that is obtained from the cylinder $P_{\Gamma_n}\times I$ by sequential shaving some facets of its bases. We show that the well-known series of polytopes (associahedra, cyclohedra, permutohedra and stellohedra) can be derived by these constructions. As a corollary we obtain inductive formulas for $\gamma$- and $h$- vectors of the mentioned series. These formulas communicate the method of differential equations developed by the first author with the method of shavings developed by the second author.
\end{abstract}

\section{Introduction}

Simple polytopes play important role in toric geometry and topology (see \cite{BR}). The classical problem of upper and lower bounds for $h$-vectors of $n$-dimensional simple polytopes with fixed number of facets is solved in \cite{Ba1}, \cite{Ba2} and \cite{Mc}.

Nowadays there appeared an important subclass of simple polytopes - Delzant polytopes. For every Delzant polytope $P^n$ there exists a Hamiltonian toric manifold $M^{2n}$ such that $P^n$ is the image of the moment map (see \cite{CdS}, \cite{D}). Davis-Januszkiewicz theorem (see \cite{DJ}) states that odd Betti numbers $b_{2i-1}(M^{2n})$ are zero and even Betti numbers $b_{2i}(M^{2n})$ are equal to components $h_i(P^n)$ of the $h$-vector of $P^n$. So, the problem of upper and lower bounds for $h$-vectors of Delzant polytopes become actual, because its solution gives upper and lower bounds for Betti numbers of Hamiltonian toric manifolds.

Feichtner and Sturmfels (see \cite{FS}) and Postnikov (see \cite{P}) showed that the Minkowski sum
of some set of regular simplices is a simple polytope if this set
satisfies certain combinatorial conditions identifying it as a buiding
set. The resulting family of simple polytopes was called nestohedra in
\cite{PRW} because of their connection to nested sets considered by De Concini
and Procesi (see \cite{DP}) in the context of subspace arrangements. Note that from results of \cite{FS} directly follows that nestohedra are Delzant polytopes. Special cases of
building sets are vertex sets of connected subgraphs in a given graph; the
corresponding nestohedra called graph-associahedra by Carr and Devadoss were first~studied~in~\cite{CD},\cite{DJS},\cite{P},\cite{TL},\cite{Ze}.

The main goal of this paper is to establish upper and lower bounds for $f$-,$g$-,$h$- and $\gamma$-vectors of graph-associahedra and some its important subclasses.

From \cite{FM} we know that if $B_1\subseteq B_2$ for connected building sets, then $P_{B_2}$ is obtained from $P_{B_1}$ by sequential shaving some faces, consequently, $h_i(P_{B_1})\leq h_i(P_{B_2})$. Therefore, $h_i(\Delta^n)\leq h_i(P_B)\leq h_i(Pe^n)$ for every $n$-dimensional nestohedron $P_B$ and these bounds are unimprovable.

In the combinatorics of simple polytopes especially interested is $\gamma$-vector. Using \cite{Bu1} and definitions of $g$-,$h$- and $f$-vectors one can prove
that componentwise inequality $\gamma(P_1)\leq\gamma(P_2)$ for simple $n$-polytopes $P_1$ and $P_2$ implies componentwise inequalities:$\quad g(P_1)\leq g(P_2),\quad h(P_1)\leq h(P_2),\quad f(P_1)\leq f(P_2)$.

Gal's conjecture (see \cite{G}) states that flag simple polytopes have nonnegative $\gamma$-vectors. In \cite{Bu2} it was described realization of the associahedron as a polytope obtained from the standard cube by shaving faces of codimension 2. The main result of \cite{V1,V2} is that every flag nestohedron has such a realization. As a corollary it was derived that unimprovable bounds for $\gamma$-vectors of flag nestohedra are $\gamma(I^n)$ and $\gamma(Pe^n)$. That includes Gal's conjecture for flag nestohedra, since $\gamma_i(I^n)=0, i>0$.

There are remarkable series of graph-associahedra corresponding to series of graphs: associahedra $As^n$ (path graphs), cyclohedra $Cy^n$ (cyclic graphs), permutohedra $Pe^n$ (complete graphs) and stellohedra $St^n$ (star graphs). Using these series we obtain the main result of the paper:

\begin{thm*} There are following unimprovable bounds:
\begin{enumerate}
\item[1)] $\gamma_i(As^n)\leq\gamma_i(P_{\Gamma_{n+1}})\leq\gamma_i(Pe^n)$ for any connected graph $\Gamma_{n+1}$ on $[n+1]$;
\item[2)] $\gamma(Cy^n)\leq\gamma_i(P_{\Gamma_{n+1}})\leq\gamma_i(Pe^n)$ for any Hamiltonian graph $\Gamma_{n+1}$ on $[n+1]$;
\item[3)] $\gamma_i(As^n)\leq\gamma_i(P_{\Gamma_{n+1}})\leq\gamma_i(St^n)$ for any tree $\Gamma_{n+1}$ on $[n+1]$.
\end{enumerate}
\end{thm*}

The last part was predicted in \cite[Conjecture 14.1]{PRW}, where it was calculated $\gamma$-vectors of trees on 7 nodes and it was noticed that
more branched and forked trees give polytopes with higher $\gamma$-vectors.

We use the constructions that for an $(n-1)$-dimensional graph-associahedron $P_{\Gamma_n}$ produce the\\ $n$-dimensional graph-associahedron $P_{\Gamma_{n+1}}$ that is obtained from the cylinder $P_{\Gamma_n}\times I$ by sequential shaving some facets of its bases. We show that the mentioned series of polytopes (associahedra, cyclohedra, permutohedra and stellohedra) can be derived by these constructions. As a corollary we obtain inductive formulas for $\gamma$- and $h$- vectors of the above series. These formulas communicate the method of differential equations developed in \cite{Bu1} with the method of shavings developed in \cite{V1,V2}.

\section{Face polynomials}
The convex $n$-dimensional polytope $P$ is called \emph{simple} if its every vertex belongs to exactly $n$ facets.\medskip\\
Let $f_i$ be the number of $i$-dimensional faces of an $n$-dimensional polytope $P$. The vector $(f_0,\ldots,f_n)$ is called the $f$-vector of $P$. The $F$-polynomial of $P$ is defined by:
\begin{equation*}
F(P)(\alpha,t)=\alpha^n+f_{n-1}\alpha^{n-1}t+\dots +f_1\alpha t^{n-1}+f_0 t^n.
\end{equation*}
The $h$-vector and $H$-polynomial of $P$ are defined by:
\begin{equation*}
H(P)(\alpha,t)=h_0\alpha^n+h_1\alpha^{n-1}t+\dots+h_{n-1}\alpha t^{n-1}+h_n t^n=F(P)(\alpha-t,t).
\end{equation*}
The $g$-vector of a simple polytope $P$ is the vector $(g_0,g_1,\dots,g_{[\frac{n}{2}]})$, where $g_0=1,\quad g_i=h_i-h_{i-1}, i>0$.\medskip\\
The Dehn-Sommerville equations (see \cite{Zi}) state that $H(P)$ is symmetric for any simple polytope. Therefore, it can be represented as a polynomial of $a=\alpha+t$ and $b=\alpha t$:
\begin{equation*}
H(P)=\sum\limits_{i=0}^{[\frac{n}{2}]}\gamma_i(\alpha t)^i(\alpha+t)^{n-2i}.
\end{equation*}
The $\gamma$-vector of $P$ is the vector $(\gamma_0,\gamma_1,\dots,\gamma_{[\frac{n}{2}]})$. The $\gamma$-polynomial of $P$ is defined by:
\begin{equation*}
\gamma(P)(\tau)=\gamma_0+\gamma_1\tau+\dots+\gamma_{[\frac{n}{2}]}\tau^{[\frac{n}{2}]}.
\end{equation*}

\begin{lem}\label{g}
Let $\gamma_i(P_1)\leq\gamma_i(P_2), i=0,\ldots,[\frac{n}{2}]$, where $P_1$ and $P_2$ are simple $n$-polytopes, then
\begin{enumerate}
\item[1)] $g_i(P_1)\leq g_i(P_2)$;
\item[2)] $h_i(P_1)\leq h_i(P_2)$;
\item[3)] $f_i(P_1)\leq f_i(P_2)$.
\end{enumerate}
\end{lem}
\begin{proof}
The following formula for simple $n$-polytopes (see \cite{Bu1}) implies part 1).
\begin{equation*}
g_i(P)=(n-2i+1)\sum_{j=0}^i \frac{1}{n-i-j+1}\binom{n-2j}{i-j}\gamma_j(P).
\end{equation*}
Next formulas derived from definitions of $g$- and $h$-vectors show that 1) implies 2) and 2) implies 3).
\begin{equation*}
h_i(P)=\sum_{j=0}^i g_j(P);\qquad f_i(P)=\sum_{j=i}^n \binom{j}{i} h_{n-j}(P).
\end{equation*}
\end{proof}

\section{Nestohedra and graph-associahedra}
In this section we state well-known facts about nestohedra. They can be found in \cite{FS},\cite{P},\cite{Ze}.
\begin{not*}
    By $[n]$ and $[i,j]$ denote the sets $\{1,\ldots,n\}$ and $\{i,\ldots,j\}$.
\end{not*}

\begin{defn}
A collection $B$ of nonempty subsets of $[n+1]$ is called a \emph{building set} on $[n+1]$ if the following conditions hold:
\begin{enumerate}
\item[1)]If $S_1,S_2\in B$ and $S_1\cap S_2\neq\emptyset$, then $S_1\cup S_2\in B$;
\item[2)]$\{i\}\in B$ for every $i\in[n+1]$.
\end{enumerate}
The building set $B$ is \emph{connected} if $[n+1]\in B$.
\end{defn}

\noindent
The \emph{restriction} of the building set $B$ to $S\in B$ is the following building set on $[|S|]$:
\begin{equation*}
B|_S=\{S'\in B\colon S'\subseteq S\}.
\end{equation*}
The \emph{contraction} of the building set $B$ along $S\in B$ is the following building set on $[n+1-|S|]$:
\begin{equation*}
B/S=\{S'\subseteq [n+1]\setminus S\colon S'\in B\textmd{ or } S'\cup S\in B\}=\{S'\setminus S, S'\in B\}.
\end{equation*}

\begin{defn}
Let $\Gamma$ be a graph with no loops or multiple edges on the node set $[n+1]$. The \emph{graphical building set} $B(\Gamma)$ is the collection of nonempty subsets $S\subseteq[n+1]$ such that the induced subgraph $\Gamma|_S$ on the node set $S$ is connected.
\end{defn}

\begin{rem}
Building set $B(\Gamma)$ is connected if and only if $\Gamma$ is connected.
\end{rem}

\begin{rem}\label{rem}
Let $\Gamma$ be a connected graph on $[n+1]$ and $S\in B(\Gamma)$, then $B|_S$ and $B/S$ are both graphical building sets corresponding to connected graphs $\Gamma|_S$ and $\Gamma/S$.
\end{rem}
\noindent
Let $M_1$ and $M_2$ be subsets of $\mathbb{R}^n$. The \emph{Minkowski sum} of $M_1$ and $M_2$ is the following subset of $\mathbb{R}^n$:
\begin{equation*}
M_1+M_2=\{x\in\mathbb{R}^n\colon x=x_1+x_2, x_1\in M_1, x_2\in M_2\}.
\end{equation*}
If $M_1$ and $M_2$ are convex polytopes, then so is $M_1+M_2$.

\begin{defn}
Let $e_i$ be the endpoints of the basis vectors of $\mathbb{R}^{n+1}$. Define the \emph{nestohedron} $P_B$ corresponding to the building set $B$ as following
\begin{equation*}
P_B=\sum_{S\in B}\Delta^S\text{, where }\Delta^S=\conv\{e_i, i\in S\}.
\end{equation*}
If $B(\Gamma)$ is a graphical building set, then $P_{\Gamma}=P_{B(\Gamma)}$ is called a \emph{graph-associahedron}.
\end{defn}

\begin{ex*} Here we especially interested by the following series of graph-associahedra:
\begin{itemize}
\item Let $L_{n+1}$ be the path graph on $[n+1]$, then the polytope $P_{L_{n+1}}$ is called \emph{associahedron} (Stasheff polytope) and denoted by $As^n$;
\item Let $C_{n+1}$ be the cyclic graph on $[n+1]$, then the polytope $P_{C_{n+1}}$ is called \emph{cyclohedron} (Bott-Taubes polytope) and denoted by $Cy^n$;
\item Let $K_{n+1}$ be the complete graph on $[n+1]$, then the polytope $P_{K_{n+1}}$ is called \emph{permutohedron} and denoted by $Pe^n$;
\item Let $K_{1,n}$ be the complete bipartite graph on $[n+1]$, then the polytope $P_{K_{1,n}}$ is called \emph{stellohedron} and denoted by $St^n$.
\end{itemize}
\end{ex*}

The simple $n$-polytope $P\subset \mathbb{R}^n$ is called a \emph{Delzant polytope} if for every its vertex $p$ there exist integer vectors parallel to the edges meeting at $p$ and forming a $\mathbb{Z}$ basis of $\mathbb{Z}^n\subset\mathbb{R}^n$.

\begin{prop}
Let $B$ be a connected building set on $[n+1]$. Then, $\dim P_B=n$ and $P_B$ can be realized as a Delzant polytope. Particularly, every nestohedron is simple.
\end{prop}

The convex polytope $P$ is called \emph{flag} if any collection of its pairwise intersecting facets has a nonempty intersection.

\begin{prop}
Every graph-associahedron is flag.
\end{prop}

To understand the combinatorics of $P_B$ we need the following statement.

\begin{prop}\label{criteria}
Let $B$ be a connected building set on $[n+1]$. Then, elements $S$ of $B\setminus[n+1]$ are in bijection with facets (denoted by $F_S$) of $P_B$, which are combinatorially equivalent to $P_{B|_S}\times P_{B/S}$.\\
Facets $F_{S_1},\ldots ,F_{S_k}$ have a nonempty intersection if and only if the following conditions hold:
\begin{enumerate}
\item[1)] For any $S_i, S_j$ we have $S_i\subset S_j$, or $S_i\supset S_j$, or $S_i\cap S_j=\emptyset$;
\item[2)] For any $S_{i_1},\dots,S_{i_p}$ such that $S_{i_j}\cap S_{i_l}=\emptyset$ we have $S_{i_1}\sqcup\ldots\sqcup S_{i_p}\notin B$.
\end{enumerate}
\end{prop}

\begin{not*}
Eventually, we will write "facet $S$" meaning "facet $F_S$".
\end{not*}

\section{Technique}

\subsection{The $\gamma$-vectors of flag nestohedra} We need the next results from \cite{V1,V2}.

\begin{thm}\label{Gal}
The $\gamma$-vector of any flag nestohedron has nonnegative entries.
\end{thm}

\begin{thm}\label{leq}
If $B_1$ and $B_2$ are connected building sets on $[n+1]$, $B_1\subseteq B_2$, and $P_{B_i}$ are flag, then
\begin{equation*}
\gamma_i(P_{B_1})\leq\gamma_i(P_{B_2}).
\end{equation*}
\end{thm}

\subsection{Shavings}
Here we describe used machinery from \cite{V1,V2}, which is based on \cite{FM}.

\begin{cons*}[\emph{Decomposition} of $S\in B_1$ by elements of $B_0$]
Let $B_0$ and $B_1$ be connected building sets on $[n+1]$, $B_0\subset B_1$, and $S\in B_1$. Let us call the decomposition of $S$ by elements of $B_0$ the representation $S=S_1\sqcup\ldots \sqcup S_k, S_j\in B_0$ such that $k$ is minimal among such disjoint representations.

It is easy to check that the decomposition exists and is unique.
\end{cons*}

The next statement can be extracted from \cite[Theorem 4.2]{FM} and also the direct proof in accepted terms is given in \cite[Lemma 5]{V1}.
\begin{thm}\label{fm}
Let $B_0$ and $B_1$ be connected building sets on $[n+1]$ and $B_0\subset B_1$. The set $B_1$ is partially ordered by inclusion. Let us number all the elements of $B_1\setminus B_0$ by indexes $i$ in such a way that $i\leq i'$ provided $S^i\supseteq S^{i'}$.

Then, $P_{B_1}$ is obtained from $P_{B_0}$ by sequential shaving faces $G^i=\bigcap_{j=1}^{k_i} F_{S_j^i}$ that correspond to decompositions $S^i=S_1^i\sqcup\ldots\sqcup S_{k_i}^i\in B_1\setminus B_0$ starting from $i=1$ (i.e. in reverse inclusion order).
\end{thm}

According to \cite{V1,V2}, if $P_{B_0}$ and $P_{B_1}$ are flag, then we can change the order of shavings (compare to the last theorem) in such a way that only faces of codimension 2 will be shaved off. This type of shavings increases the $\gamma$-vector in case of flag nestohedra.

\begin{prop}\label{shave}\emph{(cf. \cite[Proposition 6]{V1})}
Let the polytope $Q$ be obtained from the simple polytope $P$ by shaving the face $G$ of codimension 2, then
\begin{enumerate}
\item[1)]$\gamma(Q)=\gamma(P)+\tau\gamma(G)$;
\item[2)]$H(Q)=H(P)+\alpha t H(G)$.
\end{enumerate}
\end{prop}

\subsection{Substitution of building sets}

\begin{cons*}[N.~Erokhovets]
Let $B,B_1,\ldots ,B_{n+1}$ be connected building sets on $[n+1],[k_1],\ldots ,[k_{n+1}]$. Define the connected building set $B(B_1,\ldots ,B_{n+1})$ on $[k_1]\sqcup\ldots\sqcup[k_{n+1}]=[k_1+\ldots +k_{n+1}]$ consisting of the elements $S^i\in B_i$ and $\bigsqcup\limits_{i\in S}[k_i]$, where $S\in B$.
\end{cons*}

\begin{prop}[N.~Erokhovets]\label{eq1}
Let $B,B_1,\ldots ,B_{n+1}$ be connected building sets on $[n+1],[k_1],\ldots ,[k_{n+1}]$, and $B'=B(B_1,\ldots ,B_{n+1})$. Then $P_{B'}$ is combinatorially equivalent to $P_B\times P_{B_1}\times\dots\times P_{B_{n+1}}$ and the following mapping $\varphi\colon (B\setminus[n+1])\bigsqcup\limits_{i=1}^{n+1}(B_i\setminus[k_i])\to B'\setminus [k_1+\ldots +k_{n+1}]$ defines the facet correspondence.
\begin{equation*}
\varphi(S)=\begin{cases} S&\text{, if }S\in B_i;\\\bigsqcup\limits_{i\in S}[k_i]&\text{, if }S\in B.\end{cases}
\end{equation*}
\end{prop}

\begin{ex*}
Let $B,B_1,B_2$ be building sets $\{\{1\},\{2\},\{1,2\}\}$, which correspond to the interval $I$. Let us show what is $B(B_1,B_2)$. We substitute $B_1=\{\{1\},\{2\},\{1,2\}\}$ as $a$ and $B_2=\{\{3\},\{4\},\{3,4\}\}$ as $b$ to the building set $\{\{a\},\{b\},\{a,b\}\}$ and obtain the building set $B'$ on $[4]$ consisting of $\{i\}, \{1,2\}, \{3,4\}, [4]$. Here we reordered $B,B_1,B_2$ to make them not intersecting.

The facet correspondence is:
\begin{align*}
\{1\}\in B_1&\mapsto\{1\}\in B';\qquad\{2\}\in B_1\mapsto\{2\}\in B';\\
\{3\}\in B_2&\mapsto\{3\}\in B';\qquad\{4\}\in B_2\mapsto\{4\}\in B';\\
\{a\}\in B_2&\mapsto\{1,2\}\in B';\quad\{b\}\in B_2\mapsto\{3,4\}\in B'.
\end{align*}
\end{ex*}

\begin{not*}
If, for example, building sets $B_1,\ldots, B_n$ are $\{1\},\ldots,\{n\}$, then we will write $B(1,2,\ldots,n, B_{n+1})$ instead of $B(\{1\},\{2\},\ldots,\{n\}, B_{n+1})$ to simplify the notations.
\end{not*}

\section{Inductive formulas}
Let $J$ be the building set $\{\{1\},\{2\},\{1,2\}\}$, which corresponds to the interval $P_J=I$.

\begin{cons}\label{ind}
Let $\Gamma_n$ be a connected graph on $[n]$ and $V\subseteq [n]$ induces the complete subgraph of $\Gamma_n$. Set $\Gamma_{n+1}=\Gamma_{n}\cup\{n+1\}\cup \{n+1, V\}$, i.e. the vertexes $V$ are adjacent to the new vertex $\{n+1\}$.

According to proposition \ref{eq1}, the building set $B_1=J(B(\Gamma_n), n+1)=B(\Gamma_n)\cup \{n+1\}\cup[n+1]$ corresponds to $P_{B_1}=P_{\Gamma_n}\times I$: the bottom and top bases of $P_{\Gamma_n}\times I$ are $[n],\{n+1\}\in B_1$; the side facets are $S\in B(\Gamma_n)\setminus[n]\subset B_1$. Thus, side facets $S$ are naturally identified with facets of $P_{\Gamma_n}$.

$B(\Gamma_{n+1})\setminus B_1=\{S\sqcup \{n+1\}, S\in \mathcal{S}\}$, where $\mathcal{S}=\{S\in B(\Gamma_n)\setminus[n]\colon S\cap V\neq\emptyset\}$. By Theorem~\ref{fm}, $P_{\Gamma_{n+1}}$ is obtained from $P_{\Gamma_n}\times I$ by shaving intersections of the top base $F_{\{n+1\}}$ with the side facets $F_S$ for $S\in\mathcal{S}$. Since the top base doesn't change after shaving its facets, the shaved off faces are exactly $P_{\Gamma_n|_S}\times P_{\Gamma_n/S}, S\in\mathcal{S}$. By proposition \ref{shave}, we have:
\begin{align}
\gamma(P_{\Gamma_{n+1}})&=\gamma(P_{\Gamma_n})+\tau\sum_{S\in\mathcal{S}}\gamma(P_{\Gamma_n|_S})\gamma( P_{\Gamma_n/S});\label{formulas}\\
H(P_{\Gamma_{n+1}})&=(\alpha+t)H(P_{\Gamma_n})+\alpha t\sum_{S\in\mathcal{S}}H(P_{\Gamma_n|_S})H( P_{\Gamma_n/S}).
\end{align}
\end{cons}

We required that $V$ induces the complete subgraph of $\Gamma_n$, because in this case every element of $B(\Gamma_{n+1})\setminus B_1$ has decomposition consisting of two elements ($S\sqcup \{n+1\}$, where $\{n+1\},S\in B_1$) and we know the combinatorial type of shaved off faces, which have codimension 2.

\subsection{Associahedra}
Let us apply the construction \ref{ind} to the path graph $L_n$. Here $V=\{n\}$ and $\mathcal{S}=\{[i,n], i=2,\ldots n\}$. Therefore, the shaved off faces are $As^{i-1}\times As^{n-i-1}, i=1,\ldots,n-1$ and we obtain inductive formulas:
\begin{align}
\gamma(As^n)&=\gamma(As^{n-1})+\tau\sum_{i=1}^{n-1}\gamma(As^{i-1})\gamma(As^{n-i-1});\label{associahedra}\\
H(As^n)&=(\alpha+t)H(As^{n-1})+\alpha t\sum_{i=1}^{n-1}H(As^{i-1}) H(As^{n-i-1}).
\end{align}
The inductive formulas for associahedra are equivalent to the equations:
\begin{align*}
\gamma_{As}(x) &= 1 + x \gamma_{As}(x) + \tau x^2 \gamma_{As}^2(x)&\text{, where }\qquad \gamma_{As}(x) = \sum_{n=0}^\infty \gamma(As^n)x^n;\\
H_{As}(x) &= (1 + \alpha x H_{As}(x))(1 + t x H_{As}(x))&\text{, where }\qquad H_{As}(x) = \sum_{n=0}^\infty H(As^n)x^n.
\end{align*}
The last equation is equivalent to:
\begin{equation*}
\frac{x H_{As}(x)}{(1 + \alpha x H_{As}(x))(1 + t x H_{As}(x))}=x.
\end{equation*}
Set $U(x)=x H_{As}(x)$. Then, $U(0)=0, U'(0)=1$ and
\begin{equation*}
\frac{U}{(1+\alpha U)(1+t U)}=x.
\end{equation*}
Applying the classical Lagrange Inversion Formula we obtain:
\begin{multline*}
U(x) = -\frac{1}{2\pi i}\oint_{|z| = \varepsilon}
\ln\left[1-\frac{x}{z}(1+\alpha z)(1+t z)\right]dz=\\= \sum_{n=1}^\infty\left(\frac{1}{2\pi i}\oint_{|z| = \varepsilon}\left[
\frac{(1+\alpha z)^n(1+t z)^n}{z^n}\right]dz\right) \frac{x^n}{n}=\\= \sum_{n=1}^\infty\left(
\sum_{i+j=n-1}\binom{n}{i}\binom{n}{j}\alpha^it^j \right)
\frac{x^n}{n}.
\end{multline*}
Therefore,
\begin{equation*}
H(As^n) = \frac{1}{n+1}\sum_{i+j=n} \binom{n+1}{i} \binom{n+1}{j}
\alpha^it^j = \frac{1}{n+1}\sum_{i=0}^n \binom{n+1}{i} \binom{n+1}{i+1}
\alpha^{n-i}t^i.
\end{equation*}

\begin{lem}\label{as}
For every connected graph $\Gamma_{n+1}$ on $[n+1]$ we have $\gamma_i(P_{\Gamma_{n+1}})\geq\gamma_i(As^n)$.
\end{lem}
\begin{proof}
Notice, that it is enough to prove the lemma for trees. Indeed, for every connected graph $\Gamma$ there exists a tree $T\subseteq\Gamma$ on the same nodes. Then, $B(T)\subseteq B(\Gamma)$ and we can apply Theorem \ref{leq}.

For $n=1$ there is nothing to prove.

Assume that the lemma holds for $m\leq n$. Let $\Gamma_{n+1}$ be a tree on $[n+1]$. Without loss of generality, assume that $\{n+1\}$ is adjacent only to $\{n\}$. Then, we can use the construction \ref{ind} putting $\Gamma_n=\Gamma_{n+1}\setminus\{n+1\}$ and $V=\{n\}$. For every $i\in[1,n-1]$ there exists a connected subgraph of $\Gamma_n$ on $i$ vertexes containing $\{n\}$, i.e. there exists $S\in \mathcal{S}\colon |S|=i$. Therefore, comparing \eqref{formulas} to \eqref{associahedra} and using the inductive assumption and remark \ref{rem} we obtain the lemma.
\end{proof}

\subsection{Cyclohedra}
Let $C_{n+1}$ be a cyclic graph on $[n+1]$. We apply the construction different from the construction \ref{ind}.

According to proposition \ref{eq1}, the building set $B_1=B(C_n)(1,\ldots,n-1,J(n,n+1))$ corresponds to $P_{B_1}=Cy^{n-1}\times I$: the bottom and top bases of $P_{\Gamma_n}\times I$ are $\{n\},\{n+1\}\in B_1$; the side facets are $S\in B_1\setminus[n+1]$ such that $\{n,n+1\}$ is either contained in $S$ or doesn't intersect $S$. Side facets $S$ are identified with facets of $Cy^{n-1}=P_{C_n}$ by collapsing $\{n,n+1\}$ to the point $\{n\}$.

$B(C_{n+1})\setminus B_1=\{S\sqcup\{n\},S\in \mathcal{S}_n\}\cup\{S\sqcup\{n+1\},S\in \mathcal{S}_{n+1}\}$, where $\mathcal{S}_n=\{[i,n-1], i=1,\ldots,n-1\}$ and $S_{n+1}=\{[1,i], i=1,\ldots,n-1\}$.
By Theorem~\ref{fm}, $Cy^n$ is obtained from $Cy^{n-1}\times I$ by shaving intersections of the bottom base $F_{\{n\}}$ with the side facets $F_S$ for $S\in\mathcal{S}_n$ and intersections of the top base $F_{\{n+1\}}$ with the side faces $F_S$ for $S\in\mathcal{S}_{n+1}$. Since bases don't change after shaving their facets, the shaved off faces are exactly $P_{C_n|_S}\times P_{C_n/S}, S\in\mathcal{S}_n\cup\mathcal{S}_{n+1}$, which are $As^{i-1}\times Cy^{n-i-1}, i=1,\ldots,n-1$ in the top and the same type faces in the bottom. By proposition \ref{shave}, we obtain inductive formulas:

\begin{align}
\gamma(Cy^n)&=\gamma(Cy^{n-1})+2\tau\sum_{i=1}^{n-1}\gamma(As^{i-1})\gamma(Cy^{n-i-1});\\
H(Cy^n)&=(\alpha+t)H(Cy^{n-1})+2\alpha t\sum_{i=1}^{n-1}H(As^{i-1}) H(Cy^{n-i-1}).
\end{align}
The inductive formulas for cyclohedra are equivalent to the equations:
\begin{align*}
\gamma_{Cy}(x) &= 1 + x \gamma_{Cy}(x) + \tau x^2 \gamma_{As}(x)\gamma_{Cy}(x)&\text{, where }\qquad \gamma_{Cy}(x) = \sum_{n=0}^\infty \gamma(Cy^n)x^n;\\
H_{Cy}(x) &= 1 + (\alpha+t) x H_{Cy}(x) + 2 \alpha t x^2 H_{As}(x)H_{Cy}(x)&\text{, where }\qquad H_{Cy}(x) = \sum_{n=0}^\infty H(Cy^n)x^n.
\end{align*}
Set $V(x) = x H_{Cy}(x)$. Then,
\begin{equation*}
\frac{V}{1+(\alpha+t)V+2\alpha t U V}=x.
\end{equation*}
Therefore,
\begin{equation*}
\frac{U}{(1+\alpha U)(1+t U)} = \frac{V}{1+(\alpha+t)V+2\alpha t U V}.
\end{equation*}
And we obtain:
\begin{equation*}
V = \frac{U}{1-\alpha t U^2}.
\end{equation*}
%\\dd\\dd\\d\\d\\d\\d

\subsection{Permutohedra}
Let us apply the construction \ref{ind} to the complete graph $K_n$. Here $V=[n]$ and $\mathcal{S}=2^{[n]}\setminus\{\emptyset,[n]\}$. Therefore, we shave off $\binom{n}{i}$ faces of the type $Pe^{i-1}\times Pe^{n-i-1}, i=1,\ldots,n-1$ and obtain inductive formulas:

\begin{align}
\gamma(Pe^n)&=\gamma(Pe^{n-1})+\tau\sum_{i=1}^{n-1}\binom{n}{i} \gamma(Pe^{i-1})\gamma(Pe^{n-i-1});\\
H(Pe^n)&=(\alpha+t)H(Pe^{n-1})+\alpha t\sum_{i=1}^{n-1}\binom{n}{i} H(Pe^{i-1}) H(Pe^{n-i-1}).
\end{align}
The inductive formulas for permutohedra are equivalent to the differential equations:
\begin{align*}
\frac{d\gamma_{Pe}(x)}{dx} &= 1 + \gamma_{Pe}(x) + \tau \gamma_{Pe}^2(x)&\text{, where }\qquad \gamma_{Pe}(x) = \sum_{n=0}^\infty \gamma(Pe^n)\frac{x^{n+1}}{(n+1)!};\\
\frac{dH_{Pe}(x)}{dx} &= (1 + \alpha H_{Pe}(x))(1 + t H_{Pe}(x))&\text{, where }\qquad H_{Pe}(x) = \sum_{n=0}^\infty H(Pe^n)\frac{x^{n+1}}{(n+1)!}.
\end{align*}
One can explicitly solve the last equation and obtain:
\begin{equation*}
H_{Pe}(x)=\frac{e^{\alpha x}-e^{t x}}{\alpha e^{t x}-t e^{\alpha x}}.
\end{equation*}
Let $A(n,k)=|\{\sigma\in \Sym(n)\colon \des(\sigma)=k\}|$, then
\begin{equation*}
H(Pe^n)=\sum_{i=0}^{n} A(n+1,k)\alpha^k t^{n-k}.
\end{equation*}

\subsection{Stellohedra}
Let us apply the construction \ref{ind} to the complete bipartite graph $K_{1,n-1}$ or $(n-1)$-star. Here $V=\{1\}$ and $\mathcal{S}=\{\{1\}\cup S, S\subsetneq [2,n]\}$. Therefore, we shave off $\binom{n-1}{i-1}$ faces of the type $St^{i-1}\times Pe^{n-i-1},\\i=1,\ldots,n-1$ and obtain inductive formulas:

\begin{align}
\gamma(St^n)&=\gamma(St^{n-1})+\tau\sum_{i=1}^{n-1}\binom{n-1}{i-1} \gamma(St^{i-1})\gamma(Pe^{n-i-1});\label{stellohedra}\\
H(St^n)&=(\alpha+t)H(St^{n-1})+\alpha t\sum_{i=1}^{n-1}\binom{n-1}{i-1} H(St^{i-1}) H(Pe^{n-i-1}).
\end{align}
The inductive formulas for stellohedra are equivalent to the differential equations:
\begin{align*}
\frac{d\gamma_{St}(x)}{dx} &= \gamma_{St}(x)(1 + \tau x \gamma_{Pe}(x))&\text{, where }\qquad \gamma_{St}(x) = \sum_{n=0}^\infty \gamma(St^n)\frac{x^n}{n!};\\
\frac{dH_{St}(x)}{dx} &= H_{St}(x)(\alpha+t + \alpha t x H_{Pe}(x))&\text{, where }\qquad H_{St}(x) = \sum_{n=0}^\infty H(St^n)\frac{x^n}{n!}.
\end{align*}

\begin{lem}\label{st}
For every tree $\Gamma_{n+1}$ on $[n+1]$ we have $\gamma_i(P_{\Gamma_{n+1}})\leq\gamma_i(St^n)$.
\end{lem}
\begin{proof}
For $n=1$ there is nothing to prove.

Assume that the lemma holds for $m\leq n$. Let $\Gamma_{n+1}$ be a tree on $[n+1]$. Without loss of generality, assume that $\{n+1\}$ is adjacent only to $\{n\}$. Then, we can use the construction \ref{ind} putting $\Gamma_n=\Gamma_{n+1}\setminus\{n+1\}$ and $V=\{n\}$. For every $i\in[1,n-1]$ there are no more than $\binom{n-1}{i-1}$ elements $S\in \mathcal{S}\colon |S|=i$ and for each such $S$ we have $\gamma(P_{\Gamma_n|_S})\gamma(P_{\Gamma_n/S})\leq\gamma(St^{i-1})\gamma(Pe^{n-i-1})$. Therefore, comparing \eqref{formulas} to \eqref{stellohedra} and using the inductive assumption we obtain the lemma.
\end{proof}

\section{Bounds of face polynomials}
\begin{defn}
The graph $\Gamma$ is called Hamiltonian if it contains a Hamiltonian cycle, i.e. a closed loop that visits each vertex of $\Gamma$ exactly once.
\end{defn}

Summarizing Lemmas 1-3 and Theorems 1-2 we obtain the following results:
\newpage
\begin{thm}
For any flag $n$-dimmensional nestohedron $P_B$ we have:
\begin{enumerate}
\item[1)] $\gamma_i(I^n)\leq\gamma_i(P_B)\leq\gamma_i(Pe^n)$;
\item[2)] $g_i(I^n)\leq g_i(P_B)\leq g_i(Pe^n)$;
\item[3)] $h_i(I^n)\leq h_i(P_B)\leq h_i(Pe^n)$;
\item[4)] $f_i(I^n)\leq f_i(P_B)\leq f_i(Pe^n)$.
\end{enumerate}
\end{thm}

\begin{thm}
For any connected graph $\Gamma_{n+1}$ on $[n+1]$ we have:
\begin{enumerate}
\item[1)] $\gamma_i(As^n)\leq\gamma_i(P_{\Gamma_{n+1}})\leq\gamma_i(Pe^n)$;
\item[2)] $g_i(As^n)\leq g_i(P_{\Gamma_{n+1}})\leq g_i(Pe^n)$;
\item[3)] $h_i(As^n)\leq h_i(P_{\Gamma_{n+1}})\leq h_i(Pe^n)$;
\item[4)] $f_i(As^n)\leq f_i(P_{\Gamma_{n+1}})\leq f_i(Pe^n)$.
\end{enumerate}
\end{thm}

\begin{thm}\label{cy}
For any Hamiltonian graph $\Gamma_{n+1}$ on $[n+1]$ we have:
\begin{enumerate}
\item[1)] $\gamma_i(Cy^n)\leq\gamma_i(P_{\Gamma_{n+1}})\leq\gamma_i(Pe^n)$;
\item[2)] $g_i(Cy^n)\leq g_i(P_{\Gamma_{n+1}})\leq g_i(Pe^n)$;
\item[3)] $h_i(Cy^n)\leq h_i(P_{\Gamma_{n+1}})\leq h_i(Pe^n)$;
\item[4)] $f_i(Cy^n)\leq f_i(P_{\Gamma_{n+1}})\leq f_i(Pe^n)$.
\end{enumerate}
\end{thm}

\begin{thm}
For any tree $\Gamma_{n+1}$ on $[n+1]$ we have:
\begin{enumerate}
\item[1)] $\gamma_i(As^n)\leq\gamma_i(P_{\Gamma_{n+1}})\leq\gamma_i(St^n)$;
\item[2)] $g_i(As^n)\leq g_i(P_{\Gamma_{n+1}})\leq g_i(St^n)$;
\item[3)] $h_i(As^n)\leq h_i(P_{\Gamma_{n+1}})\leq h_i(St^n)$;
\item[4)] $f_i(As^n)\leq f_i(P_{\Gamma_{n+1}})\leq f_i(St^n)$.
\end{enumerate}
\end{thm}

\begin{proof}[Proof of Theorem \ref{cy}]
Since $\Gamma_{n+1}$ is Hamiltonian, there exists a cyclic subgraph $C_{n+1}\subseteq\Gamma_{n+1}$. Therefore, $B(C_{n+1})\subseteq B(\Gamma_{n+1})$ and Theorem \ref{leq} allows to finish the proof.
\end{proof}

The mentioned bounds can be written explicitly using results about $f$-,$h$-,$g$- and $\gamma$-vectors of the considered series (cf. \cite{PRW} and \cite{Bu1}):

\begin{gather*}
    h_i(I^n)=\binom{n}{i};\quad h_i(As^n)=\frac{1}{n+1}\binom{n+1}{i}\binom{n+1}{i+1};\quad h_i(Cy^n)=\binom{n}{i}^2;\\\\
    h_i(Pe^n)=A(n+1,i);\quad h_i(St^n)=\sum_{k=i}^n\binom{n}{k}A(k,i-1), i>0;\\\\
    \gamma_i(I^n)=0, i>0;\quad\gamma_i(As^n)=\frac{1}{i+1}\binom{2i}{i}\binom{n}{2i};\quad\gamma_i(Cy^n)=\binom{n}{i,i,n-2i}.
\end{gather*}

The derived bounds for $f$-,$h$-,$g$- and $\gamma$-vectors give the bounds for the corresponding polynomials, which generating functions were obtained in \cite{Bu1}.

\small\bigskip
\textsc{Steklov Mathematics Institute, Russian Academy of Sciences, Moscow, Russia}\\
\emph{E-mail adress:} \verb"buchstab@mi.ras.ru"\bigskip\\
\textsc{Department of Mechanics and Mathematics, Moscow State University, Moscow, Russia}\\
\emph{E-mail adress:} \verb"volodinvadim@gmail.com"
\end{document}